\documentclass[12pt,leqno]{amsart}
\usepackage{amssymb,amsthm,amsmath,latexsym}
\usepackage{xcolor}
\usepackage{setspace}
\usepackage{dcpic}
\usepackage{multicol}
\usepackage{pictexwd, dcpic}
\usepackage{graphicx}
\usepackage{enumerate}
\usepackage{tikz-cd}
\usepackage[all]{xy}
\newcommand{\vfi}{\varphi}

\newtheorem{theorem}{\sc Theorem}[section]

\newtheorem{lem}[theorem]{\sc Lemma}
\newtheorem{prop}[theorem]{\sc Proposition}
\newtheorem{cor}[theorem]{\sc Corollary}

\newtheorem{rem}[theorem]{\sc Remark}
\newtheorem{defi}[theorem]{\sc Definition}
\newtheorem{ex}[theorem]{\sc Example}

\newtheorem*{thmA}{Theorem A}
\newtheorem*{thmB}{Theorem B}
\newtheorem*{thmC}{Theorem C}

\title[non-abelian tensor square]{On some series of a group Related to the Non-abelian Tensor Square of Groups}
\author[Bastos]{R. Bastos}
\address{Departamento de Matem\'atica, Universidade de Bras\'ilia, Brasilia-DF, 70910-900 Brazil}
\email{(Bastos) bastos@mat.unb.br, (Rocco) norai@unb.br }
\author[de Oliveira]{R. de Oliveira}
\address{ Instituto de Matem\'atica e Estat\'istica, Universidade Federal de Goi\'as, Goi\^ania-GO, 74690-900 Brazil}
\email{(de Oliveira) ricardo@ufg.br}
\author[Monetta]{C. Monetta}
\address{Dipartimento di Matematica, Universit\`a di Salerno, Via Giovanni Paolo II, 132 - 84084 - Fisciano (SA), Italy}
\email{(Monetta) cmonetta@unisa.it}
\author[Rocco]{N.\,R. Rocco}

\begin{document}
%%% ----------------------------------------------------------------------
\maketitle
%%% ----------------------------------------------------------------------
\begin{abstract}
Let $G$ be a group. We denote by $\nu(G)$ a certain extension of the non-abelian tensor square $G \otimes G$ by $G \times G$. In this paper we prove that the derived subgroup $\nu(G)'$ is a central product of three normal subgroups of $\nu(G)$, all isomorphic to the non-abelian tensor square $G \otimes G$. As a consequence, we describe the structure of each term of the derived and lower central series of the group $\nu(G)$.
\end{abstract}

%%%%%%%%%%%%%%%%%%%%%% SECTION

\section{Introduction}

The non-abelian tensor square $G \otimes G$ of a group $G$, as introduced by Brown and Loday \cite{BL} following Miller \cite{Miller} and Dennis \cite{Dennis}, is defined to be the group generated by all symbols $\; \, g\otimes h, \;$ with $ g,h\in G$, subject to the relations
\[
gg_1 \otimes h = ( g^{g_1}\otimes h^{g_1}) (g_1\otimes h) \quad
\mbox{and} \quad g\otimes hh_1 = (g\otimes h_1)( g^{h_1} \otimes
h^{h_1})
\]
for all $g,g_1, h,h_1 \in G$. Here and in the sequel, given $x,y \in G$ we write $x^{y}$ for the conjugate $y^{-1}xy$. For the convenience of the reader, we establish some notation. For any integer $k\geq 2$ and for any elements $x_1,\dots,x_k$ in a group $G$, 
the $\gamma_k$-commutator $[x_1,\dots,x_k]$ is defined inductively by the formulae
\[
[x_1,x_2]=x_1^{-1}x_1^{x_2}
\qquad \text{and} \qquad
[x_1,\ldots,x_k]=[[x_1,\dots,x_{k-1}],x_k].
\]
Similarly, given $k \geq 1$ and $x_1, \ldots, x_{2^k} \in G$, the $\delta_k$-commutator is inductively defined by the formulae $\delta_1(x_1,x_2)=[x_1,x_2]$ and
\[
\delta_k(x_1, \ldots, x_{2^k})=[\delta_{k-1}(x_1,\ldots,x_{2^{k-1}}), \delta_{k-1}(x_{2^{k-1}+1},\dots,x_{2^k})].
\]
The subgroup of $G$ generated by all $\gamma_k$-commutators is denoted by $\gamma_k(G)$, while the subgroup of $G$ generated by all $\delta_k$-commutators is denoted by $G^{(k)}$. The subgroups $\gamma_k(G)$ and $G^{(k)}$ are the familiar $k$th terms of the lower central series and of the derived series of $G$, respectively. In particular, $\gamma_2(G)=G^{(1)}$ is the derived subgroup $G'$ of $G$. We also recall that if $H,K$ are subgroups of $G$, then $[H,_0K]=H$ and, for $n \geq 1$, $[H,_n K] = [[H,_{n-1} K],K]$.

There are two main approaches one can adopt to understand which group $G \otimes G$ is. The first one involves the definition of biderivation \cite{BL}.

\begin{defi}\label{def:bider}
Let $G$ and $L$ be groups. A biderivation (or crossed pairing) from $G \times G$ to $L$ is a function $f \colon G \times G \to L$ such that
\begin{itemize}
    \item[(i)] $f(aa_1,b) = f(a^{a_1},b^{a_1})f(a_1,b)$   
    \item[(ii)] $f(a,bb_1) = f(a,b_1)f(a^{b_1},b^{b_1})$
\end{itemize}
for all $a,a_1,b,b_1 \in G$.
\end{defi}

Clearly the map $j \colon G \times G \to  G \otimes G$ given by $j(g,h)= g \otimes h$ is a biderivation. Any biderivation $f \colon G \times G \to L$ determines a unique homomorphism $f^* \colon G \otimes G \to L$ such that the diagram
\[
\begin{tikzcd}
G \times G \arrow[r, "j"]  \arrow[rd, "f", swap] & G \otimes G \arrow[d, "f^*"] \\
 & L
\end{tikzcd}
\]
is commutative, that is, $f^* \circ j= f$. Therefore, to determine the non-abelian tensor square via biderivations, as done by many authors \cite{Bacon,BK,BKM,BeuK,BJR,Johnson,Kappe,MM,V}, means to conjecture both the group $L$ and the biderivation $f$, which might often appear to be a significant challenge.

On the other hand, in \cite{NR1} one of the authors suggested another way to obtain the non-abelian tensor square
by making use of the following construction.
Let $G$ be a group and let $G^{\varphi}=\{g^{\varphi} \ | \ g \in G\}$ be an isomorphic copy of $G$ via $\varphi$. Then we define the group $\nu(G)$ to be
\[
\langle G \cup G^{\varphi} \ | \ [g_1,g_2^{\varphi}]^{g_3}=[g_1^{g_3},(g_2^{g_3})^{\varphi}] =[g_1,g_2^{\varphi}]^{g_3^{\varphi}}, g_i \in G\rangle.
\]
In \cite{EL} Ellis and Leonard considered a similar construction. According to  \cite[Proposition 2.6]{NR1}, the subgroup $\Upsilon_1(G)=[G,G^{\varphi}]$ of $\nu(G)$ is isomorphic to the non-abelian tensor square $G \otimes G$, which explains the motivation to study the group $\nu(G)$. Actually, the group $\nu(G)$ admits the decomposition    $\nu(G) = (\Upsilon_1(G) \cdot G) \cdot G^{\varphi}$, where the dots mean (internal) semidirect products. Moreover, thanks to the work of several mathematicians, this construction has become powerful machinery for computing and handling the non-abelian tensor square $G \otimes G$, especially when the group $G$ is polycyclic (see  \cite{BFM,BlythMorse,EN,EL,M,M09,NR2}). It is worth mentioning the existence of another group $\chi (G)$, introduced by Sidki \cite{Sidki}, which is defined by
\[
\chi(G):= \langle G \cup G^{\varphi} \mid [g,g^{\varphi}]=1, \hbox{ for all } g\in G \rangle.
\]
This group admits a normal subgroup $R(G)$ such that $ \chi(G)/ R(G) \simeq \nu(G)/\Delta(G)$, where $\Delta(G)= \langle [g,g^{\varphi}] \mid g \in G\rangle$ is the diagonal subgroup of $\nu(G)$ (see \cite[Remark 2]{NR1} for more details). It is interesting to work with these quotients because they both contain a copy of the so called non-abelian exterior square  $G \wedge G$, that is, the group $[G,G^{\varphi}]/ \Delta(G)$  (cf. \cite[p. 1984]{NR2}).

In this paper we will apply the biderivation technique to better understand the structure of the derived and lower central series of the group $\nu(G)$.

Firstly, we denote by $\Theta(G)$ the kernel of the epimorphism $\rho \colon \nu(G) \to G$ defined by $\rho(g)=g=\rho(g^{\varphi})$ for every $g \in G$. The restriction of $\rho$ to $\Upsilon_1(G)$ gives the derived map $\rho' \colon \Upsilon_1(G) \to G'$, induced by  $[g,h^{\varphi}] \mapsto [g,h]$ for all $g,h \in G$ and $\mu(G) = \ker(\rho')$ is isomorphic to the third homotopy group of the suspension of an Eilenberg–MacLane space $K(G,1)$, $\pi_3(SK(G,1))$ (see also \cite[Proposition 3.3]{BL} and \cite[Proposition 2.7]{NR2}). Here $\rho'$  corresponds to the derived map $\kappa$ of \cite{BL}. Now we define two subgroups which will play a fundamental role in the present work. Let $\Upsilon_2(G) = [\Theta(G), G]$ and $\Upsilon_3(G)= [\Theta(G), G^{\varphi}]$. The subgroups $\Upsilon_2(G)$ and $\Upsilon_3(G)$ are normal subgroups of $\nu(G)$ contained in $\Theta(G)$. They commute with each other, and their intersection is a central subgroup of $\nu(G)$ (see Lemma~\ref{lem.upsilon} for more details). Recall that a group $H$ is a central product of its normal subgroups $K_1, K_2, \ldots, K_l$ if $H=K_1 K_2 \cdots K_l$, $[K_i,K_j]=1$ for $i\neq j$ and $K_i \cap \left(\prod_{j\neq i}K_j\right) \leq Z(H)$ for all $i$. The following result provides a refined description of the derived subgroup $\nu(G)'$, showing its strong connection with the non-abelian tensor square $G \otimes G$.

\begin{thmA} \label{thm.upsilon}
Let $G$ be a group. Then 
\begin{itemize}
    \item[(a)] the subgroups $\Upsilon_2(G)$ and $\Upsilon_3(G)$ are both isomorphic to $G \otimes G$;
    \item[(b)] the derived subgroup $\nu(G)'$ is a central product of the subgroups $\Upsilon_1(G)$, $\Upsilon_2(G)$ and $\Upsilon_3(G)$.
\end{itemize}
Moreover, the group $\nu(G)'$ is isomorphic to $G' \times G' \times G'$ modulo $\mu(G)$.
\end{thmA}

As shown in \cite[Proposition 2.7]{BuenoRocco}, 
\begin{itemize}
    \item $\nu(G)^{(k)}=([G^{(k-1)}, (G^{\varphi})^{(k-1)}] \cdot G^{(k)}) \cdot (G^{\varphi})^{(k)}$ 
    \item $\gamma_{k+1}(\nu(G))=([\gamma_{k}(G), G^{\varphi}] \cdot \gamma_{k+1}(G)) \cdot \gamma_{k+1}(G^{\varphi})$ 
\end{itemize}
for every $k \geq 1$. The following results provide a convenient description for the derived and lower central series of $\nu(G)$ in terms of the subgroups $\Upsilon_1(G),\Upsilon_2(G)$ and $\Upsilon_3(G)$.

\begin{thmB}\label{thm:derived}
Let $G$ be a group and let $k \geq 0$. Then
$$ \nu(G)^{(k+1)} = \Upsilon_1(G)^{(k)} \Upsilon_2(G)^{(k)} \Upsilon_3(G)^{(k)}.$$
\end{thmB}

If $k\geq 1$, we write $A_k(G) =  [\Upsilon_1(G),_{k-1}G]$, $B_k(G) =  [\Upsilon_2(G),_{k-1}G]$ and $C_k(G) = [\Upsilon_3(G),_{k-1}G^{\varphi}].$ We establish the following related result. 

\begin{thmC}\label{thm:lowercentral}
Let $G$ be a group. If $k\geq 2$, then
\[
\gamma_{k}( \nu(G))= A_{k-1}(G) \ B_{k-1}(G) \ C_{k-1}(G).
\]
\end{thmC}

%%%%%%%%%%%%%%%%%%%%% SECTION

The paper is organized as follows. In the next section we prove Theorems A and B. The third section is devoted to the proof of Theorem C. In Section 4, we provide some applications of Theorems A and B, firstly computing the derived subgroup $\nu(G)'$ for some groups (Theorem \ref{theorem_derived_subgroup}), and then obtaining some bounds for the exponent of $\nu(G)$ and its subgroups (see Corollaries \ref{cor:exponent}--\ref{cor:Gab}). 

\section{The derived series of $\nu(G)$: biderivations and $\Upsilon_i(G)$}

By the symmetry of the defining relations of $\nu(G)$, we note that the isomorphism $\varphi$ extends uniquely to an automorphism $\Psi$ of $\nu(G)$ sending $g \mapsto g^{\varphi}$, $g^{\varphi} \mapsto g$ and $[g_1, g_2^{\varphi}] \mapsto [g_2, g_1^{\varphi}]^{-1}$ for all $g, g_1, g_2 \in G$. In particular, since $\Upsilon_2(G)=[\Theta(G), G]$ and $\Upsilon_3(G)=[\Theta(G), G^{\varphi}]$, we have that $\Upsilon_3(G)= \Psi(\Upsilon_2(G))$. We point out that the choice of these subgroups was motivated by \cite[Section 4]{Sidki}. 

\subsection{The subgroups $\Upsilon_i(G)$}

In this subsection we prove part (a) of Theorem A. 

The next result summarizes general properties of $\Upsilon_i(G)$ that will be useful in the main proofs.  

\begin{lem} \label{lem.upsilon}
Let $G$ be a group. 
\begin{itemize}
    \item[(a)] $\Upsilon_i(G)$ is normal in $\nu(G)$ for $i=1,2,3$; 
    \item[(b)] $\Upsilon_2(G)  = \langle [g,h][h,g^{\varphi}] \mid g,h\in G\rangle \leq \Theta(G)$, \\ $\Upsilon_3(G) = \langle [g^{\varphi},h^{\varphi}][h^{\varphi},g] \mid g,h\in G\rangle \leq \Theta(G)$;
    \item[(c)] \mbox{\rm (Rocco, \cite[Theorem  2.11]{NR2})} $[\Upsilon_2(G), G^{\varphi}]=1= [\Upsilon_3(G), G]$;
    \item[(d)] 
    $[\Upsilon_i(G),\Upsilon_j(G)]=1$ for $i,j=1,2,3$ and $i \neq j$.
\end{itemize}
\end{lem}

\begin{proof}

\noindent (a)  By \cite[Proposition 2.6]{NR1}, $\Upsilon_1(G)$ is a normal subgroup of $\nu(G)$. Moreover, $\Upsilon_2(G) = [\Theta(G),G]$ is normal in $\Theta(G)G=\nu(G)$, and $\Upsilon_3(G)= [\Theta(G),G^{\varphi}]$ is normal in $\Theta(G)G^{\varphi}=\nu(G)$.  

\noindent (b) Since $\Theta(G)$ is a normal subgroup of $\nu(G)$, it follows that $\Upsilon_i(G) \leq \Theta(G)$, for $i=2,3$.

Set $K = \langle [g,h][h,g^{\varphi}] \mid g,h\in G\rangle$. For any $a,b \in G$ we have 
$$ 
[a,b^{\varphi}b^{-1}] = [a,b^{-1}][a,b^{\varphi}]^{b^{-1}} = ([b,a][a,b^{\varphi}])^{b^{-1}} \in K.
$$
As $[a,b^{\varphi}b^{-1}]$ is a standard generator of $\Upsilon_2(G)$, we have $\Upsilon_2(G) \leq K$. Moreover, using the same symbols, $[b,a][a,b^{\varphi}] = [a,b^{\varphi}b^{-1}]^{b} \in \Upsilon_2(G)$, and thus $K = \Upsilon_2(G)$, as required. By symmetry, the result also follows for $\Upsilon_3(G)$.

\noindent (d) As $[\Upsilon_1(G),\Theta(G)]=1$ by \cite[Remark 2]{NR2}, it follows that $[\Upsilon_1(G), \Upsilon_j(G)]=1$ for $j=2,3$. Since $\Upsilon_3(G) \leq \Upsilon_1(G) (G')^{\varphi}$, from part (b) we have $$[\Upsilon_2(G),\Upsilon_3(G)] \leq [\Upsilon_2(G),\Upsilon_1(G)(G')^{\varphi}]=1.$$
\end{proof}

A first step towards accomplishing our goal is to show the existence of a biderivation from $G \times G$ to $\Upsilon_2(G)$.

\begin{lem}\label{lem:biderivation}
The map $\Gamma \colon G \times G \to \Upsilon_2(G)$, given by $$\Gamma(g,h) = [g,h][h,g^{\varphi}],$$ is a biderivation from $G \times G$ to $\Upsilon_2(G)$. In particular, $\Upsilon_2(G)$ is an epimorphic image of the non-abelian tensor square $G \otimes G$.
\end{lem}

\begin{proof}
Let $g,g_1,h,h_1 \in G$. Since $ \Gamma(g_1,h),  \Gamma(g^{h_1},h^h_1) \in \Upsilon_2(G)$, by Lemma~\ref{lem.upsilon} (d) we get 
\begin{eqnarray*}
\Gamma(gg_1,h) & = &  [gg_1,h][h,(gg_1)^{\varphi}] = [g,h]^{g_1}[g_1,h][h,g_1^{\varphi}][h,g^{\varphi}]^{g_1^{\varphi}} \\
 & = & [g,h]^{g_1} \Gamma(g_1,h) [h,g^{\varphi}]^{g_1} = [g,h]^{g_1} [h,g^{\varphi}]^{g_1}  \Gamma(g_1,h) \\
  & = & [g^{g_1},h^{g_1}] [h^{g_1},(g^{g_1})^{\varphi}] \Gamma(g_1,h) = \Gamma(g^{g_1},h^{g_1}) \Gamma(g_1,h),
\end{eqnarray*}

and 

\begin{eqnarray*}
\Gamma(g,hh_1) & = & [g,hh_1][hh_1,g^{\varphi}] = [g,h_1][g,h]^{h_1}[h,g^{\varphi}]^{h_1}[h_1,g^{\varphi}] \\
&=& [g,h_1] [g^{h_1},h^{h_1}][h^{h_1},(g^{h_1})^{\varphi}] [h_1,g^{\varphi}] = [g,h_1] \Gamma(g^{h_1},h^{h_1}) [h_1,g^{\varphi}] \\
&=& [g,h_1] [h_1,g^{\varphi}] \Gamma(g^{h_1},h^{h_1}) = \Gamma(g,h_1) \Gamma(g^{h_1},h^{h_1}).
\end{eqnarray*}
We deduce that $\Gamma$ is a biderivation and so there exists a unique homomorphism $\Gamma^* \colon G \otimes G \to \Upsilon_2(G)$ which makes the diagram
\[
\begin{tikzcd}
G \times G \arrow[r, "j"]  \arrow[rd, "\Gamma", swap] & G \otimes G \arrow[d, "\Gamma^*"] \\
 & \Upsilon_2(G)
\end{tikzcd}
\]
commutative, that is, $\Gamma^* \circ j=\Gamma$.
\end{proof}

Since $G \otimes G \simeq \Upsilon_1(G)$, we can define $\Gamma^* \colon \Upsilon_1(G) \to \Upsilon_2(G)$ by
 \[
  \Gamma^* \left(\prod_{i=1}^s [x_i,y_i^{\varphi}]^{\varepsilon_i} \right)= \prod_{i=1}^s ([x_i,y_i][y_i,x_i^{\varphi}])^{\varepsilon_i},
 \]
 where $x_i,y_i \in G$ and $\varepsilon_i \in \{-1,1\}$ for $i=1, \ldots, s$. By Lemma~\ref{lem:biderivation}, $\Gamma^*$ is an epimorphism.

For the readers' convenience we recall the following morphisms related to $\nu(G)$ (see \cite{NR2}):  
\begin{enumerate} [$\ast$]
    \item the epimorphism $\rho \colon \nu(G) \to G$ defined by $\rho(g)=g=\rho(g^{\varphi})$ for every $g \in G$.
    \item the isomorphism $\Psi \colon \nu(G) \to \nu(G)$ defined by sending $g \mapsto g^{\varphi}$, $g^{\varphi} \mapsto g$, $[g_1, g_2^{\varphi}] \mapsto [g_2, g_1^{\varphi}]^{-1}$ for all $g, g_1, g_2 \in G$.
\end{enumerate}
 
In the following lemma we describe the map $\Gamma^*$ in terms of the maps $\rho$ and $\Psi$.
 
\begin{lem}\label{lem:tec}
Let $s$ be a positive integer and let $\alpha_i=[x_i, y_i^{\varphi}]^{\varepsilon_i}$ with $x_i,y_i \in G$ and $\varepsilon_i \in \{-1,1\}$ for $i=1, \ldots, s$. Then
\[
 \Gamma^* \left(\prod_{i=1}^s \alpha_i \right)=\left(\prod_{i=1}^s \Psi(\alpha_{s-i+1})^{-1}\right) \left(\prod_{i=1}^s \rho(\alpha_i)\right).
\]

\end{lem}

\begin{proof}
First of all observe that for $a,b \in G$
\[
[a,b][b,a^{\varphi}]= \rho([a,b^{\varphi}]) \Psi([a,b^{\varphi}])^{-1} = \Psi([a,b^{\varphi}])^{-1} \rho([a,b^{\varphi}]).
\]
Moreover, using the defining relations of $\nu(G)$, for every $x,y \in G$ 
\[
\Psi([a,b^{\varphi}])^{\rho([x,y^{\varphi}])}= \Psi([a,b^{\varphi}])^{\Psi([x,y^{\varphi}])}.
\]
Then, by an induction argument on $s\geq 1$ we have 
\begin{align*}
\displaystyle
\Gamma^* \left(\prod_{i=1}^s \alpha_i \right) &= \prod_{i=1}^s \Psi(\alpha_i)^{-1}\rho(\alpha_i) = \Psi(\alpha_1)^{-1}\rho(\alpha_1) \prod_{i=2}^s \Psi(\alpha_i)^{-1}\rho(\alpha_i) \\
& = \Psi(\alpha_1)^{-1}\rho(\alpha_1) \left(\prod_{i=2}^{s} \Psi(\alpha_{s-i+2})^{-1}\right) \left(\prod_{i=2}^{s} \rho(\alpha_i)\right) \\
& = \Psi(\alpha_1)^{-1} \left(\prod_{i=2}^{s} \Psi(\alpha_{s-i+2})^{-\rho(\alpha_1)^{-1}}\right) \left(\prod_{i=1}^{s} \rho(\alpha_i)\right) \\
& = \Psi(\alpha_1)^{-1} \left(\prod_{i=2}^{s} \Psi(\alpha_{s-i+2})^{-\Psi(\alpha_1)^{-1}}\right) \left(\prod_{i=1}^{s} \rho(\alpha_i)\right) \\
& = \left(\prod_{i=1}^{s} \Psi(\alpha_{s-i+1})^{-1}\right) \left(\prod_{i=1}^{s} \rho(\alpha_i)\right) \\
\end{align*}
and we are done.
\end{proof}

Now, we are in a position to prove Theorem A~(a).

\begin{proof}[Proof of Theorem A~(a)]
As $\Upsilon_3(G) \simeq \Upsilon_2(G)$, we will show only that $\Upsilon_2(G) \simeq\Upsilon_1(G)$. We will prove that $\Gamma^* \colon \Upsilon_1(G) \to \Upsilon_2(G)$ is an isomorphism by showing that it is injective.

Let $\prod_{i=1}^s \alpha_i \in \Upsilon_1(G)$ such that $\Gamma^*(\prod_{i=1}^s \alpha_i)=1$.
By Lemma~\ref{lem:tec} we have 
\[
1= \prod_{i=1}^s \Psi(\alpha_{s-i+1})^{-1} \prod_{i=1}^s \rho(\alpha_i)
\]
which implies
\[
\prod_{i=1}^s \Psi(\alpha_i)= \prod_{i=1}^s \rho(\alpha_i) \in G' \cap \Upsilon_1(G).
\]
Since $G' \cap \Upsilon_1(G)=1$, it follows that $1= \prod_{i=1}^s \Psi(\alpha_i)= \Psi(\prod_{i=1}^s \alpha_i)$, which yields $\prod_{i=1}^s \alpha_i=1$.
\end{proof}

\subsection{The structure of the derived series of $\nu(G)$}
In this subsection we prove part (b) of Theorem A and  Theorem B. 

\begin{prop}\label{prop:produtonormais}
Let $G$ be a group. Then
\begin{itemize}
\item[(a)] $\mu(G) \leq \Upsilon_i(G)$, $i \in \{1,2,3\}$;
\item[(b)] $\mu(G) = \Upsilon_i(G) \cap \Upsilon_j (G)$, $i,j\in \{1,2,3\}$ and $i\neq j$; 
\item[(c)] The derived subgroup $\nu(G)'$ is equal to $ \Upsilon_1(G) \Upsilon_2(G) \Upsilon_3(G)$. 
\end{itemize}
\end{prop}

\begin{proof}
(a) 
Since $\mu(G) = \{\prod_{i=1}^{n} [x_i,y_i^{\varphi}]^{\varepsilon_i} \mid \prod_{i=1}^{n} [x_i,y_i]^{\varepsilon_i}=1, \varepsilon_i \in \{-1,1\}\}$, consider $\prod_{i=1}^{n} [x_i,y_i^{\varphi}]^{\varepsilon_i}$ with  $\prod_{i=1}^{n} [x_i,y_i]^{\varepsilon_i}=1$. On the one hand, we have that 
\[
\prod_{i=1}^{n} [y_{n-i+1},x_{n-i+1}]^{\varepsilon_i}=1, \]
and so
\[
\prod_{i=1}^{n} [x_i,y_i^{\varphi}]^{\varepsilon_i}= \prod_{i=1}^{n} [x_i,y_i^{\varphi}]^{\varepsilon_i} \prod_{i=1}^{n} [y_{n-i+1}, x_{n-i+1}]^{\varepsilon_i}.
\]
On the other hand, Lemma~\ref{lem:tec} implies that
\[ 
\Gamma^* \left(\prod_{i=1}^{n} [y_{n-i+1},x_{n-i+1}^{\varphi}]^{\varepsilon_i}\right) = \prod_{i=1}^{n} [x_i,y_i^{\varphi}]^{\varepsilon_i} \prod_{i=1}^{n} [y_{n-i+1},x_{n-i+1},]^{\varepsilon_i},
\]
and thus $\mu(G) \leq \Gamma^*(\Upsilon_1(G))= \Upsilon_2(G)$.
As $\Psi(\mu(G))=\mu(G)$, clearly $\mu(G) \leq \Upsilon_3(G)$.

\noindent (b) Firstly, by item (a) we have
\[ 
\mu(G) \leq \Upsilon_1(G) \cap \Upsilon_i(G) \leq \Upsilon_1(G) \cap \Theta(G) = \mu(G).
\]
Secondly, since $\Upsilon_2(G) \leq \Upsilon_1(G)G'$ and $\Upsilon_3(G) \leq \Upsilon_1(G)(G')^{\varphi}$, it follows that
\begin{align*}
  \Upsilon_2(G) \cap \Upsilon_3(G) & \leq \Upsilon_1(G)G' \cap \Upsilon_1(G)(G')^{\varphi} \\ &=\Upsilon_1(G)(G' \cap \Upsilon_1(G)(G')^{\varphi})=\Upsilon_1(G),  
\end{align*}
where in the last equality we use the fact that $\nu(G)'$ is the internal semidirect product $(\Upsilon_1(G) \cdot G') \cdot (G')^{\varphi}$ (see Theorem 3.1 of \cite{NR1}).
As $ \Upsilon_2(G) \cap \Upsilon_3(G) \leq \Theta(G)$, we conclude that
\[
\mu(G) \leq \Upsilon_2(G) \cap \Upsilon_3(G) \leq  \Theta(G) \cap \Upsilon_1(G)=\mu(G).
\] 

\noindent (c) Set $A= \Upsilon_1(G) \Upsilon_2(G) \Upsilon_3(G)$. It is clear that $A \leq \nu(G)'$, hence, we only need to prove the converse. Since $\nu(G)'=\Upsilon_1(G)G' (G')^{\varphi}$, it is sufficient to show that $G'$ and $(G')^{\varphi}$ are contained in $A$.

For every $g,h \in G$ we have $[g,h][h,g^{\varphi}] \in \Upsilon_2(G) \leq A$, and $[h,g^{\varphi}] \in \Upsilon_1(G)\leq A$. Therefore $[g,h] \in A$ for every $g,h \in G$, which implies $G' \leq A$. Similarly, from  $[g^{\varphi},h^{\varphi}][h^{\varphi},g] \in \Upsilon_3(G) \leq A$, and $[h^{\varphi},g] \in \Upsilon_1(G)\leq A$, it follows that $(G')^{\varphi} \leq A$. This concludes the proof. 
\end{proof}

Let us now prove Theorem A~(b).

\begin{proof}[Proof of Theorem A~(b)] 

By Lemma \ref{lem.upsilon}, for distinct $i,j,k \in \{1,2,3\}$ we have $$[\Upsilon_i(G),\Upsilon_j(G)]=1.$$ Now we show that for distinct $i,j,k \in \{1,2,3\}$ we have   
$$ 
\Upsilon_i(G) \cap \Upsilon_j(G)\Upsilon_k(G)=\mu(G),
$$
which implies $\Upsilon_i(G) \cap \Upsilon_j(G)\Upsilon_k(G) \leq Z(\nu(G))$.

Firstly observe that $\mu(G) \leq \Upsilon_i(G)$ for $i \in {1,2,3}$. Thus
\[
\mu(G) \leq \Upsilon_1(G) \cap \Upsilon_2(G)\Upsilon_3(G) \leq \Upsilon_1(G) \cap \Theta(G) = \mu(G),
\]
which implies that $\mu(G) = \Upsilon_1(G) \cap \Upsilon_2(G)\Upsilon_3(G)$. Now, for distinct $j,k \in \{2,3\}$, by part (b) of Proposition~\ref{prop:produtonormais} it is sufficient to prove that $\Upsilon_j(G) \cap \Upsilon_1(G) \Upsilon_k(G) \leq \Upsilon_j(G) \cap \Upsilon_k(G)$. Firstly,
we have that $\Upsilon_j(G) \cap \Upsilon_1(G) \Upsilon_k(G) \leq \Upsilon_j(G)$, and moreover by Dedekind's Modular Law it follows that
\begin{align*}
  \mu(G) \leq  \Upsilon_j(G) \cap \Upsilon_1(G) \Upsilon_k(G) & \leq \Theta(G) \cap \Upsilon_1(G) \Upsilon_k(G) \\ & = (\Theta(G) \cap \Upsilon_1(G)) \Upsilon_k(G) = \Upsilon_k(G),
\end{align*}
and hence $\Upsilon_j(G) \cap \Upsilon_1(G) \Upsilon_k(G)= \mu(G)$, as required.

Now, it remains to prove that $\nu(G)'$ is isomorphic to $G'\times G'\times G'$ modulo $\mu(G)$. As proved earlier $\Upsilon_l(G) \cap \Upsilon_m(G)\Upsilon_n(G) = \mu(G)$ for distinct $l,m,n \in \{1,2,3\}$. Now, by Lemma \ref{lem.upsilon}~(d) and Proposition \ref{prop:produtonormais}~(c), the derived subgroup $\nu(G)'$ is the product of the commuting subgroups $\Upsilon_1(G)$, $\Upsilon_2(G)$ and $\Upsilon_3(G)$, which implies that
\[
\frac{\nu(G)'}{\mu(G)}=\frac{\Upsilon_1(G)\Upsilon_2(G)\Upsilon_3(G)}{\mu(G)} \simeq \frac{\Upsilon_1(G)}{\mu(G)} \times \frac{\Upsilon_2(G)}{\mu(G)} \times\frac{\Upsilon_3(G)}{\mu(G)}.
\]
Since $\Upsilon_1(G) / \mu(G)\simeq G'$, we only need to show that $\Upsilon_i(G) /\mu(G) \simeq G'$, for $i=2,3$. Since $\Upsilon_2(G) /\mu(G) \simeq \Upsilon_3(G) /\mu(G)$, let us show that $\Upsilon_2(G) /\mu(G) \simeq G'$. Firstly, consider the map $\xi:\Upsilon_2(G) \to G'$ such that $$\xi\left(\prod_{i=1}^s ([g_i,h_i][h_i,g_i^{\varphi}])^{\varepsilon_i} \right)=\prod_{i=1}^s [g_i,h_i]^{\varepsilon_i}.$$ This map is the composition of the derived map $\rho':\Upsilon_1(G) \to G'$ and the isomorphism $({\Gamma^*})^{-1} :\Upsilon_2(G) \to \Upsilon_1(G)$. Since the kernel of $\xi$ is the subgroup $\mu(G)$, we have $({\Gamma^*})^{-1}(\mu(G))=\mu(G)$, and the proof is concluded.
\end{proof}

The proof of Theorem B is now straightforward. 

\begin{proof}[Proof of Theorem B]
If $k=0$ the assertion follows from Theorem A~(b). By Lemma \ref{lem.upsilon}~(d), the result follows when $k \geq 1$.
\end{proof}

It is well-known that if $G$ is nilpotent (respectively, Engel), then so is $G \otimes G$ (see \cite{BKM,BJR,V}). Now, using Theorem A~(b), we obtain similar results for the derived subgroup $\nu(G)'$.

\begin{cor} \label{cor:nu(G)'}
Let $G$ be a group.
\begin{itemize}
    \item[(a)] If $G'$ is nilpotent of class $c$, then $\nu(G)'$ is nilpotent of class at most $c+1$;
    \item[(b)] If $G'$ is $n$-Engel for $n \geq 1$, then $\nu(G)'$ is at most $(n+1)$-Engel;
    \item[(c)] If $G'$ is locally nilpotent, then so is $\nu(G)'$.
\end{itemize}
\end{cor}

We conclude the section by enriching the subgroup lattice  of $\nu(G)$ presented in \cite[p. 1984]{NR2} with our new normal subgroups. For a better understanding, we will need the following.

\begin{prop}
Let $G$ be a group. Then 
\begin{itemize}
\item[(a)] $\nu(G)/(\Upsilon_1(G)\Theta(G))$ is isomorphic to $G^{ab}$; 
\item[(b)] $\Upsilon_1(G)\Theta(G)/\nu(G)' \simeq \Theta(G)/(\Upsilon_2(G)\Upsilon_3(G))\simeq G^{ab}$;
\end{itemize}
where $G^{ab}$ is the abelianization of $G$, namely, $G/G'$. 
\end{prop}

\begin{proof}
(a) It is sufficient to observe that $\nu(G)/\Theta(G)$ is isomorphic to $G$ and  $(\Upsilon_1(G)\Theta(G))/ \Theta(G)) \simeq \Upsilon_1(G) / \mu(G) \simeq G'$.

\noindent (b) First observe that
\[
\frac{\Theta(G)\Upsilon_1(G)}{\nu(G)'} = \frac{\Theta(G)\nu(G)'}{\nu(G)'} \simeq \frac{\Theta(G)}{\Theta(G) \cap \nu(G)'} = \frac{\Theta(G)}{\Upsilon_2(G)\Upsilon_3(G)}.
\]
Set $H=\Upsilon_2(G)\Upsilon_3(G)$ and notice that $x^{-1}x^{\varphi} \in H$ for any $x \in G'$. Then the map $f \colon G^{ab} \to \Theta(G)/H$ given by $gG' \to g^{-1}g^{\varphi}H$ is a well-defined injective map. By Lemma \ref{lem.upsilon} we have $[H, \nu(G)] \leq H$, which shows that $f$ is a homomorphism. As $f$ is surjective, it is an isomorphism.  
\end{proof}

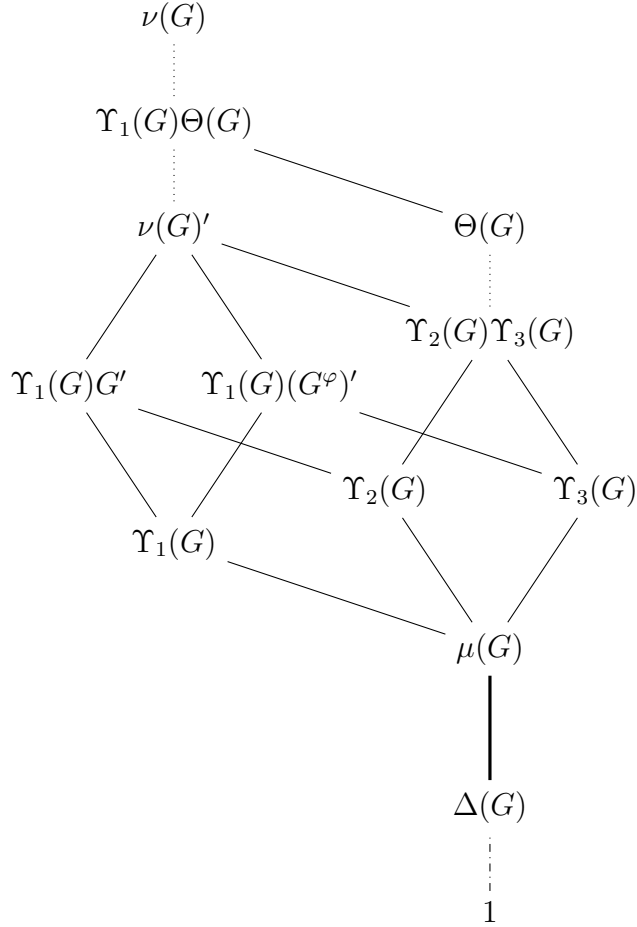
\begin{center}
\begin{tikzpicture}[scale=.7]
  \node (one) at (0,3) {$\nu(G)$};
  \node (a) at (0,1) {$\Upsilon_1(G) \Theta(G)$};
  \node (b) at (6,-1) {$\Theta(G)$};
  \node (c) at (0,-1) {$\nu(G)'$};
  \node (d) at (6,-3) {$\Upsilon_2(G) \Upsilon_3(G)$};
  \node (e) at (6,-9) {$\mu(G)$};
  \node (f) at (4,-6) {$\Upsilon_2(G)$};
  \node (g) at (8,-6) {$\Upsilon_3(G)$};
  \node (h) at (-2,-4) {$\Upsilon_1(G) G'$};
  \node (i) at (2,-4) {$\Upsilon_1(G) (G^{\varphi})'$};
  \node (l) at (0,-7) {$\Upsilon_1(G)$};
 \node (m) at (6,-12) {$\Delta(G)$};
  \node (n) at (6,-14) {$1$};
  \draw [dotted](one)  -- (a);
  \draw (a)-- (b);
  \draw [dotted](b)-- (d);
  \draw (d)--(g) -- (e);
  \draw [dotted](a) -- (c);
  \draw (c)-- (h) --(l) -- (e);
  \draw (c) -- (d) -- (f) -- (e);
  \draw (c)-- (i) -- (l);
  \draw (h) -- (f);
  \draw (g) -- (i);
  \draw [very thick](m) -- (e);
  \draw [dash dot](n) -- (m);
\end{tikzpicture}
\end{center}

\begin{rem}
For the sake of completeness, we summarize some lattice properties of these new subgroups. Indeed, in the previous diagram, all the subgroups are normal in $\nu(G)$, and the lines of the same type identify the same quotient group: ``dotted" $\simeq G^{ab}$; ``thin" $\simeq G'$; ``thick" $\simeq H_2(G)$ is the second homology group of $G$ (see \cite[Section 2]{NR2} and \cite{Miller}); ``dash dot" $\simeq \Delta(G)$.
\end{rem}

\section{Lower central series of $\nu(G)$}

The following basic properties are consequences of 
the defining relations of $\nu(G)$ and the commutator rules (see \cite[Section 2]{NR1} and \cite[Lemma 1.1]{BFM} for more details).

\begin{lem} 
\label{basic.nu}
The following relations hold in $\nu(G)$, for all 
$g, h, x, y \in G$.
\begin{itemize}
\item[(a)] $[g, h^{\varphi}]^{[x, y^{\varphi}]} = [g, h^{\varphi}]^{[x, 
y]}$; 
\item[(b)] $[g, h^{\varphi}, x^{\varphi}] = [g, h, x^{\varphi}] = [g, 
h^{\varphi}, x] = [g^{\vfi}, h, x^{\vfi}] = [g^{\vfi}, h^{\vfi}, x] = 
[g^{\vfi}, 
h, x]$.
\end{itemize}
\end{lem}

We will denote by $\Gamma_k(G)$ the group $G$ when $k=1$ and the set $\{[g_1,g_2,\ldots,g_k] \mid g_j \in G\}$ when $k \geq 2$. Recall that $A_k(G) =  [\Upsilon_1(G),_{k-1}G]$, $B_k(G) =  [\Upsilon_2(G),_{k-1}G]$ and $C_k(G) = [\Upsilon_3(G),_{k-1}G^{\varphi}].$ 

\begin{prop}\label{prop:abc-structure}
Let $G$ be a group and $k$ a positive integer. Then 
\begin{itemize}
    \item[(a)] The subgroups $A_k(G)$,$B_k(G)$ and $C_k(G)$ are normal in $\nu(G)$;
    \item[(b)] $B_k(G) \leq \Upsilon_2(G)$ and $C_k(G) \leq \Upsilon_3(G)$;
    \item[(c)] $A_k(G) = [\gamma_k(G),G^{\varphi}]$;
    \item[(d)] $B_k(G) = \langle [c,x][x,c^{\varphi}] \mid x \in G, c \in \Gamma_k(G)\rangle$;
    \item[(e)] $C_k(G) = \langle [c^{\varphi},x^{\varphi}][x^{\varphi},c] \mid x \in G, c \in \Gamma_k(G)\rangle$.
\end{itemize}
\end{prop}
\begin{proof}
(a) By \cite[Proposition 2.5]{NR1}, $A_k(G) = [\gamma_k(G),G^{\varphi}]$ is a normal subgroup of $\nu(G)$. From Lemma \ref{lem.upsilon} (c) we have $[\Upsilon_2(G),G^{\varphi}]=1$, which implies $B_k(G)=[\Upsilon_2(G),_{k-1} \nu(G)]$. By symmetry, $C_k(G)=[\Upsilon_3(G),_{k-1} \nu(G)]$, and clearly they are both normal in $\nu(G)$. 

\noindent (b) By Lemma \ref{lem.upsilon}~(a), $\Upsilon_i(G)$ is normal in $\nu(G)$, for $i=1,2,3$. From this we deduce that $B_k(G)\leq \Upsilon_2(G)$ and  $C_k(G)\leq \Upsilon_3(G)$. 

\noindent (c) We argue by induction on $k$. If $k=1$, then $A_1(G)=\Upsilon_1(G)$. 

Assume $k \geq 1$ and that the formula holds for $k$. Then, by Lemma \ref{basic.nu}~(b), we have \[
\begin{array}{ccl}
A_{k+1}(G) & = &  [\Upsilon_1(G),_k G] = [[\Upsilon_1(G),_{k-1} G],G] = [[\gamma_{k}(G),G^{\varphi}],G] \\
 & = & [[\gamma_{k}(G),G],G^{\varphi}] = [\gamma_{k+1}(G),G^{\varphi}],
\end{array} 
\]
and the assertion follows.

\noindent (d) If $k=1$, then $B_1(G) = \Upsilon_2(G)$ and the equality holds. Therefore, let $k \geq 1$ and assume the statement is  true for $k$.
%$B_k=\langle [c,x][x,c^{\varphi}] \mid x \in G, c \in \Gamma_k(G)\rangle$. 
%For $k=2$,
Using the defining relations of $\nu(G)$ and Lemma \ref{basic.nu} we deduce that for any $c \in \Gamma_k(G)$ and $x,y \in G$ we have
\begin{eqnarray*}
[[c,x][x,c^{\varphi}],y] & = &  [[c,x],y]^{[x,c^{\varphi}]}[x,c^{\varphi},y] = ([[c,x],y][y,[c^{\varphi},x]])^{[x,c^{\varphi}]} \\ 
 & = & ([[c,x],y][y,([c,x])^{\varphi}])^{[x,c^{\varphi}]} \\
 & = & [[c,x],y][y,([c,x])^{\varphi}],
\end{eqnarray*}
as $[\Theta(G),\Upsilon_1(G)]=1$.
If we denote $\langle [c,x][x,c^{\varphi}] \mid x \in G, c \in \Gamma_{k+1}(G)\rangle$ by $L$, the previous computation shows that $B_{k+1}(G)=L$, and we are done.

\noindent (e) Applying the isomorphism $\Psi$ to item (d), item (e) follows.
\end{proof}

\begin{proof}[Proof of Theorem C]
We argue by induction on $k \geq 2$. If $k=2$ the result follows from Proposition \ref{prop:produtonormais} (c). Assume $k \geq 2$ and $\gamma_k(\nu(G))=A_{k-1}(G)B_{k-1}(G)C_{k-1}(G)$. According to  Lemma \ref{lem.upsilon}~(d) and Proposition \ref{prop:abc-structure}~(a)--(b), we deduce that $A_k(G), B_k(G), C_k(G)$ are normal subgroups commuting pairwise, and so 
\begin{align*}
\gamma_{k+1}(\nu(G)) & =  [\gamma_k(\nu(G)),\nu(G)]\\
 & = [A_{k-1}(G)B_{k-1}(G)C_{k-1}(G), \nu(G)] \\
                    & =   [A_{k-1}(G), \nu(G)] [B_{k-1}(G), \nu(G)]  [C_{k-1}(G), \nu(G)]. \\
\end{align*}

Let $x \in \nu(G)$ and $a \in A_{k-1}(G)$. Then
$$ 
[a,x] =a^{-1}a^x = a^{-1}a^{\rho(x)}  = [a,\rho(x)] \in [A_{k-1}(G),G], 
$$
 which shows that  $[A_{k-1}(G),\nu(G)] = [A_{k-1}(G),G]$.
 
 According to the defining relations of $\nu(G)$ and Lemma \ref{lem.upsilon}~(c)--(d), we deduce that $B_{k-1}(G)$ commutes with $\Upsilon_1(G) \cdot G^{\varphi}$, implying 
$$
[B_{k-1}(G), \nu(G)] = [B_{k-1}(G), (\Upsilon_1(G)\cdot G) \cdot G^{\varphi}] = [B_{k-1}(G),G]. 
$$
Finally, since $\Psi(B_{k-1}(G)) = C_{k-1}(G)$, we also get $[C_{k-1}(G),\nu(G)] = [C_{k-1}(G),G^{\varphi}]$, which completes the proof.  
\end{proof}

\begin{rem}
We point out that if $N$ is a subgroup of $G$, one can consider in $\nu(G)$ the subgroups $[N,G^{\varphi}]$, $[G,N^{\varphi}]$, $\langle [n,x][x,n^{\varphi}] \mid n \in N, x \in G\rangle$ and $\langle [n^{\varphi},x^{\varphi}][x,n^{\varphi}] \mid n \in N, x \in G\rangle$. Arguing as in the proof of Proposition \ref{prop:abc-structure}~(d), it can be proved that these subgroups are normal in $\nu(G)$ if the subgroup $N$ is normal in $G$. 
Choosing $N$ to be the term $G^{(i)}$ of the derived series or the term $\gamma_k(G)$ of the lower central series of $G$ yields the subgroups $\Upsilon_i(G)$, $A_k(G),B_k(G)$ and $C_k(G)$ of $\nu(G)$, respectively.
\end{rem}

%%%%%%%%%%%%%%%%%%%%%%%%%%%%%. SECTION

\section{Applications}

\subsection{The derived subgroup $\nu(G)'$}

This section starts with the computation of the derived subgroup $\nu(G)'$ for groups $G$ for which the non-abelian tensor square $G \otimes G$ has a particular decomposition.

\begin{theorem} \label{theorem_derived_subgroup}
Let $G$ be a group. Suppose that the non-abelian tensor square $G \otimes G$ is isomorphic to $\mu(G) \times G'$. Then $\nu(G)'$ is isomorphic to $\mu(G) \times G' \times G'\times G'$. 
\end{theorem}
\begin{proof}
By Theorem A~(a), for every $i \in \{1,2,3\}$ there exist $D_i \leq \Upsilon_i(G)$  such that
$\Upsilon_i(G) = \mu(G)  D_i$ with $D_i \simeq G'$ and $D_i \cap \mu(G)= 1$. Now, for distinct $i,j,k \in \{1,2,3\}$, we have  $D_i \cap \Upsilon_j(G) \Upsilon_k(G) \leq D_i$, and 
\[
D_i \cap \Upsilon_j(G) \Upsilon_k(G) \leq \Upsilon_i(G) \cap \Upsilon_j(G) \Upsilon_k(G) = \mu(G), 
\]
which implies $D_i \cap \Upsilon_j(G) \Upsilon_k(G)=1$.
By Lemma \ref{lem.upsilon}~(d), $[D_i,D_j] \leq [\Upsilon_i(G),\Upsilon_j(G)]=1$ for $i \neq j$. Arguing as in the proof of Theorem A~(b), we conclude that $D_i$ is a normal subgroup of $\nu(G)'$ for $i \in \{1,2,3\}$. Consequently, Theorem A~(b) implies

\begin{eqnarray*}
\nu(G)' & = & \Upsilon_1(G) \Upsilon_2(G) \Upsilon_3(G) \\ 
 & = &  (\mu(G)  D_1) (\mu(G)  D_2) (\mu(G) D_3) \\
  & = &  \mu(G)  D_1  D_2  D_3 \\
 & \simeq &  \mu(G) \times G' \times G' \times G' .
\end{eqnarray*}
\end{proof}

Although the hypothesis of the previous result seems to be very restrictive, it is actually satisfied by several kinds of groups. In the next results we will apply Theorem~\ref{theorem_derived_subgroup} to compute the group $\nu(G)'$ for some infinite metacyclic groups \cite[Theorem 4.4]{BK}, for free groups of finite rank \cite[Proposition 6]{BJR}, and some linear groups \cite[Theorem~(i) and (iii)]{H}.

\begin{cor}
Let $F_n$ be a free group of rank $n$. Then the derived subgroup $\nu(F_n)'$ is isomorphic to $ \mathbb{Z}^{n(n+1)/2} \times F'_n \times F'_n \times F'_n$. 
\end{cor}

\begin{proof}
By \cite[Proposition 6]{BJR}, the non-abelian tensor square $\Upsilon_1(F_n)$ is isomorphic to $\mu(F_n) \times F'_n$, where $\mu(F_n) = \mathbb{Z}^{n(n+1)/2}$, and the result follows from Theorem \ref{theorem_derived_subgroup}. 
\end{proof}

In \cite{H}, Hannebauer considered the non-abelian tensor square of some linear groups. Combining Theorem \ref{theorem_derived_subgroup} with \cite[Theorem~(i) and (iii), p. 31; Corollary~(i) and (iii), p. 33]{H}, one obtains  

\begin{lem} \label{lem:linear}
Let $\mathbb{F}_q$ be a finite field with $q$ elements, $q \neq 2,3,4,9$. Then 
\begin{enumerate} [(a)]
    \item $\mu(SL_2(\mathbb{F}_q))=1$ and $\Upsilon_1(SL_2(\mathbb{F}_q)) \simeq SL_2(\mathbb{F}_q)$. 
    \item $\mu(GL_2(\mathbb{F}_q)) \simeq \mathbb{F}^{\ast}_{q}$ and $\Upsilon_1(GL_2(\mathbb{F}_q)) \simeq \mathbb{F}^{\ast}_{q} \times SL_2(\mathbb{F}_q)$
\end{enumerate}
\end{lem}

Now, we obtain the following related result. 

\begin{cor}
Let $\mathbb{F}_q$ be a finite field with $q$ elements, $q \neq 2,3,4,9$. Then 
\begin{enumerate} [(a)]
    \item The derived subgroup $\nu(SL_2(\mathbb{F}_q))'$ is isomorphic to $\left( SL_2(\mathbb{F}_q) \right)^3$. 
    \item The derived subgroup $\nu(GL_2(\mathbb{F}_q))'$ is isomorphic to $\left( SL_2(\mathbb{F}_q) \right)^3 \times \mathbb{F}^{\ast}_q$.
\end{enumerate}
\end{cor}
\begin{proof}
\begin{enumerate} [(a)]
    \item By Lemma \ref{lem:linear}~(a),  $$\Upsilon_1(SL_2(\mathbb{F}_q)) \simeq \mu(SL_2(\mathbb{F}_q)) \times SL_2(\mathbb{F}_q) = SL_2(\mathbb{F}_q).$$ By Theorem \ref{theorem_derived_subgroup}, the derived subgroup $\nu(SL_2(\mathbb{F}_q))'$ is isomorphic to $SL_2(\mathbb{F}_q) \times SL_2(\mathbb{F}_q) \times SL_2(\mathbb{F}_q)$.   
    \item By Lemma \ref{lem:linear}~(b), $$\Upsilon_1(GL_2(\mathbb{F}_q)) \simeq \mu(SL_2(\mathbb{F}_q)) \times SL_2(\mathbb{F}_q) = \mathbb{F}^{\ast}_q \times SL_2(\mathbb{F}_q). $$ By Theorem \ref{theorem_derived_subgroup}, the derived subgroup  $\nu(GL_2(\mathbb{F}_q))'$ is isomorphic to $\mathbb{F}^{\ast}_q \times SL_2(\mathbb{F}_q) \times SL_2(\mathbb{F}_q) \times SL_2(\mathbb{F}_q)$.
\end{enumerate}
\end{proof}

Finally, we handle a metacyclic case.
\begin{cor}
Let $G=\langle a,b \mid     b^n=1, [a,b]=a^2\rangle$, an   infinite metacyclic group. Then $\nu(G)' \simeq C_{(n,4)} \times C_2 \times C_n \times \mathbb{Z}^3$, where here $C_k$ denotes the cyclic group of order $k$ and $(n,4) = \gcd(n,4)$.
\end{cor}
\begin{proof}
By \cite[Theorems 4.4 and 5.2]{BK},  $\mu(G)$ is isomorphic to $C_{(n,4)} \times C_2 \times C_n$, the derived subgroup $G'$ is infinite cyclic and the non-abelian tensor square $\Upsilon_1(G)$ is isomorphic to $C_{(n,4)} \times C_2 \times C_n \times \mathbb{Z}$. The result follows from Theorem \ref{theorem_derived_subgroup}. 
\end{proof}

\subsection{The exponent of $\nu(G)$ and some of its sections}

Bounds for the exponent of the non-abelian tensor square and related constructions were considered by a number of authors \cite{BdMGM,BdMGN,Ellis,M,M09}. We present new bounds for the exponent of the group $\nu(G)$ and its sections.  

An immediate consequence of Theorems B and C is the following. 

\begin{cor}\label{cor:exponent}
Let $G$ be a finite $p$-group.
\begin{itemize}
    \item[(a)] If $k\geq 0$ then $\exp(\nu(G)^{(k+1)})=\exp(\Upsilon_1(G)^{(k)}).$
    \item[(b)]  If $r\geq 2$ then $\exp(\gamma_r(\nu(G))) = \exp(A_{r-1}(G))$. 
\end{itemize}
\end{cor}
\begin{proof}
(a) By Theorem B, $\nu(G)^{(k+1)} = \Upsilon_1(G)^{(k)}\Upsilon_2(G)^{(k)}\Upsilon_3(G)^{(k)}$. In particular, $\exp(\Upsilon_1(G)^{(k)})$ divides $\exp(\nu(G)^{(k+1)})$. As $\Upsilon_i(G)\simeq \Upsilon_j(G)$ and $[\Upsilon_i(G),\Upsilon_j(G)]=1$ for $i \neq j$, we have $\exp(\nu(G)^{(k+1)})$ divides $\exp(\Upsilon_1(G)^{(k)})$ and the result follows.

(b) By Theorem C, $\gamma_r(\nu(G)) = A_{r-1}(G)B_{r-1}(G)C_{r-1}(G)$. Now, arguing as in part (a), we deduce that $\exp(\gamma_r(\nu(G))) = \exp(A_{r-1}(G))$.   
\end{proof}

Let $p$ be a prime and let $G$ be a $p$-group of order $p^n$ and nilpotency class $c$. Then the {\it coclass} of $G$ is the integer $r(G)=n-c$. In \cite{BdMGM} the following bounds have been obtained.

\begin{theorem}[Theorem 1.4 of \cite{BdMGM}]
Let $p$ be a prime and $G$ a $p$-group of nilpotency class $c$. Let $n=\lceil \log_{p}(c+1)\rceil$. Then $\exp([G,G^{\varphi}])$ divides $\exp{(G)}^n$.
\end{theorem}

\begin{cor}[Corollary 1.6 of \cite{BdMGM}]
Let $p$ be a prime and $G$ a $p$-group of coclass $r$. 
\begin{enumerate}
   \item If $p\geq 3$ then $\exp(M(G))$ and $\exp(\mu(G))$ divide $ \exp(G)^{r+1}$;
    \item If $p=2$ then $\exp(M(G))$ and $\exp(\mu(G))$ divide $ \exp(G)^{r+4}$.
\end{enumerate}
\end{cor}

As a consequence, the following corollary provides new bounds for the exponent of $\nu(G)'$.

\begin{cor}  \label{cor:p-groups}
Let $p$ be a prime and $G$ a $p$-group of nilpotency class $c$ and coclass $r$. Let $n=\lceil \log_{p}(c+1)\rceil$. 
\begin{itemize}
    \item[(a)] If $p=2$ then $\exp(\nu(G)')$ divides $ {\exp(G)}^{\min\{n,r+4\}}$.
    \item[(b)] If $p>2$ then $\exp(\nu(G)')$ divides ${\exp(G)}^{\min\{n, r+1\}}$.
\end{itemize}
\end{cor}

Let $\Delta(G)$ be the subgroup of $\nu(G)$ given by $\Delta(G) = \langle [g,g^{\varphi}] \mid  g \in G\rangle$. Following \cite{EN}, denote by $\tau(G)$ the factor group $\nu(G)/\Delta(G)$. In \cite{EN} the subgroup $[G,G^{\varphi}]_{\tau} = [G,G^{\varphi}]/\Delta(G)$ is shown to be isomorphic with the non-abelian exterior square $G \wedge G$. Now, combining Corollary \ref{cor:exponent} and \cite[Theorem 1.3 (iii)]{BFM} one obtains

\begin{cor}\label{cor:Gab}
Let $p$ be an odd prime and $G$ a finite $p$-group. Then $\exp(\nu(G))$ divides $\max \{\exp(\Delta(G)),\exp(G \wedge G)\} \cdot \exp(G^{ab})$.
\end{cor}

\begin{proof}
By \cite[Proposition 2.7~(i)]{BuenoRocco}, the derived subgroup $(\nu(G))' = (\Upsilon_1(G) \cdot  G') \cdot (G')^{\varphi}$ and so,  $\exp(\nu(G))$ divides $\exp(\nu(G)')  \exp(G^{ab})$. By Corollary \ref{cor:exponent},  $\exp(\Upsilon_1(G)) = \exp(\nu(G)')$. Consequently, the exponent of $\nu(G)$ divides $\exp(\Upsilon_1(G)) \exp(G^{ab})$. By \cite[Theorem 1.3 (iii)]{BFM},  $\Upsilon_1(G) \simeq (G \wedge G) \times \Delta(G)$, which completes the proof. 
\end{proof}

We show by mean of an example that the divisibility condition obtained in Corollary~\ref{cor:Gab} can be achieved as an equality, which demonstrates that the bound is the best possible.

\begin{ex}
Indeed, if $G$ is the SmallGroup(27,3) in GAP's Library \cite{GAP4}, then $\exp(\nu(G))=\max \{\exp(\Delta(G)),\exp(G \wedge G)\} \cdot \exp(G^{ab}).$
\end{ex}

%%%%%%%%%%%%%%%%%%%%%%%%%%%%%. Acknowledgements

\section*{Acknowledgements}
The authors are very grateful to the referee who has carefully read the manuscript. Their comments were valuable for the improvement of the present version. This work was partially supported by DPI/UnB and FAPDF-Brazil. The third author was supported by the ``National Group for Algebraic and Geometric Structures, and their Applications" (GNSAGA - INdAM).

%%%%%%%%%%%%%%%%%%%%%%%%%%%%%. BIBLIOGRAPHY

\end{document}